\documentclass{amsart}

\usepackage{amsmath,amsfonts,amsthm,amssymb,verbatim,enumerate,graphicx}
\usepackage{color}
\usepackage{subfig}
\usepackage[flushmargin]{footmisc}
\newtheorem{theorem}{Theorem}[section]
\newtheorem{thmx}{Theorem}

\newtheorem{lemma}[theorem]{Lemma}
\newtheorem{proposition}[theorem]{Proposition}

\newtheorem{definition}[theorem]{Definition}

\newtheorem*{remark}{Remark}
\newtheorem*{remarks}{Remarks}
\numberwithin{equation}{section}
\newcommand{\tef}{transcendental entire function}
\newcommand\qfor{\quad\text{for }}
\newcommand \C{\mathbb{C}}
\newcommand \N{\mathbb{N}}
\newcommand \R{\mathbb{R}}
\newcommand \D{\mathbb{D}}
\def\B{\mathcal{B}}
\def\Blog{\mathcal{B}_{log}}

\def\H{\mathbb{H}}
\def\endpoints{E}
\def\escapingendpoints{\widetilde{E}}

\def\meanderingendpoints{E_{M}}
\def\escapingandmeanderingendpoints{E_{Q}}
\makeatletter
\def\blfootnote{\xdef\@thefnmark{}\@footnotetext}
\makeatother
\begin{document}
%
%
%
%
\title[The topology of the set of non-escaping endpoints]{The topology of the set of non-escaping endpoints}
\author{Vasiliki Evdoridou, \, David J. Sixsmith}
\address{Institute of Mathematics\\ Polish Academy of Sciences\\
Warsaw 00--956\\
Poland}
\email{vevdoridou@impan.pl}
\address{Dept. of Mathematical Sciences \\
	 University of Liverpool \\
   Liverpool L69 7ZL\\
   UK \\ ORCiD: 0000-0002-3543-6969}
\email{djs@liverpool.ac.uk}
%
%
%
%
\begin{abstract}
There are several classes of transcendental entire functions for which the Julia set consists of an uncountable union of disjoint curves each of which joins a finite endpoint to infinity. Many authors have studied the topological properties of this set of finite endpoints. It was recently shown that, for certain functions in the exponential family, there is a strong dichotomy between the topological properties of the set of endpoints which escape and those of the set of endpoints which do not escape. In this paper, we show that this result holds for large families of functions in the Eremenko-Lyubich class. We also show that this dichotomy holds for a family of functions, outside that class, which includes the much-studied Fatou function defined by $f(z) := z + 1+ e^{-z}.$ Finally, we show how our results can be used to demonstrate that various sets are spiders' webs, generalising results such as those in \cite{Fweb}.
\end{abstract}
\maketitle
%
%
%
%
\section{Introduction}
Let $f$ be a transcendental entire function. The set of points $z \in \C$ for which the iterates $(f^n)_{n \in \N}$ forms a normal family in some neighbourhood of $z$ is called the \textit{Fatou set} $F(f)$. The complement of $F(f)$ is the \textit{Julia set} $J(f)$. For an introduction to the properties of these sets, see \cite{bergweiler93} and \cite{milnor06}.

There are large classes of \tef s for which the Julia set consists of an uncountable union of disjoint curves each of which joins a finite endpoint to infinity. This topological structure is known as a \textit{Cantor bouquet}; see \cite{brushing} for a definition and detailed study of this configuration. 

The study of the topological properties of the endpoints of these curves began with Mayer \cite{mayer90}, who considered the endpoints for the functions $f_a(z)=e^z+a$, $a \in \C$. He proved the surprising fact that if $a<-1$, then the set of all endpoints of $f_a$ is totally separated, but its union with infinity is a connected set. Here we say that $X \subset \C$ (or $\widehat{\C}$) is \textit{totally separated} if for any two points $a, b \in X$ there exists a relatively open and closed set $U \subset X$ such that $a \in U$ and $b \notin U$. As observed in \cite{LasseNada}, it is now known that Mayer's result holds for \textit{all} Cantor bouquets.

It is natural to ask whether analogous properties hold for subsets of the set of endpoints. It follows from recent results in \cite{LasseNada} and \cite{LasseVasso} that, for many values of $a \in \C$, the following dichotomy holds for the escaping and non-escaping endpoints of the function $f_a$ (precise definitions are given later in this introduction):
\begin{enumerate}[(I)]
\item the union of the set of \emph{escaping} endpoints with $\{\infty\}$ is connected;\label{lbl:escaping}
\item the union of the set of \emph{non-escaping} endpoints with $\{\infty\}$ is totally separated.\label{lbl:nonfastescaping}
\end{enumerate}

Our principal goal in this paper is to show that a strong version of this dichotomy holds for significantly larger classes of functions. We use a novel proof technique which combines ideas from \cite{Fweb} and \cite{LasseVasso} with estimates from \cite{DevHairs}, and a conjugacy result from \cite{Rigidity}.

In order to state these results precisely we require some definitions. First we define several sets of endpoints of interest. For a {\tef} $f$ with a Cantor bouquet Julia set, we write $\endpoints(f)$ for the set of endpoints. The \textit{escaping set} of $f$ is defined by $$I(f) := \{z \in \C: f^n(z)\to \infty\;\text{as}\;n \to \infty\}.$$ We then let $\escapingendpoints(f) := \endpoints(f) \cap I(f)$ denote the set of \textit{escaping} endpoints.  The endpoints that do not belong to $I(f)$ are called \textit{non-escaping} endpoints.

We are particularly interested in the set of \textit{meandering} endpoints, which we denote by $\meanderingendpoints(f)$. This set consists of the endpoints which do not belong to the fast escaping set $A(f)$. The set $A(f)$ is a subset of $I(f)$ which was introduced by Bergweiler and Hinkkanen in \cite{semiconjugation}, and is defined as follows. First, for $R>0$, we define the \textit{maximum modulus function} by $$M(R, f) := \max_{|z| = R} |f(z)|.$$ We then set 
\[
A(f):=\{z \in \C:\;\text{there exists}\;\ell \in \N\;\text{such that}\;|f^{n+\ell}(z)| \geq M^n(R,f),\;\text{for}\;n \in \N\},
\]
where $R>0$ is such that $M(r,f)>r$ for $r \geq R$. It is known \cite[Theorem 2.2]{Fast} that this definition is independent of the choice of $R$. Note that in this definition $M^n(R, f)$ denotes $n$ iterations of $M(R, f)$ with respect to the variable $R$.

We next define the classes of functions which are of interest. The set of finite singular values, denoted by $S(f)$, is the closure of the set of all critical and asymptotic values of $f$; equivalently, it is the smallest closed set $S$ with the property that $f : \C \setminus f^{-1}(S) \to \C \setminus S$ is a covering map. The well-studied \textit{Eremenko-Lyubich class} $\mathcal{B}$ consists of those transcendental entire functions $f$ such that $S(f)$ is bounded (see \cite{EandL}). A {\tef} in the class $\B$ is \emph{hyperbolic} if every singular value lies in the basin of an attracting periodic cycle; see \cite{classBornotclassB} for this and other, equivalent, definitions. A function $f$ is of \textit{disjoint type} if it is hyperbolic and $F(f)$ is connected. Clearly this implies that for functions of disjoint type the Fatou set consists of one component which is an immediate attracting basin. 

We are now able to state our results. In most cases, these are arranged as pairs of theorems: the first stating that, for some class of functions, $J(f)$ is a Cantor bouquet with the property that the set of meandering endpoints coincides with the set $J(f) \setminus A(f)$; the second stating a dichotomy for subsets of these endpoints, which implies the dichotomy \eqref{lbl:escaping}/\eqref{lbl:nonfastescaping}. We begin with the following, which can quickly be deduced from \cite[Theorem 1.5]{brushing}, \cite[Theorem 1.2]{DevHairs} and \cite[Theorem 4.7]{RRRS}. Recall that a transcendental entire function is of \emph{finite order} if
$$\limsup_{r \to \infty} \frac{\log \log M(r,f)}{\log r} < \infty.$$ 
\begin{thmx}
\label{theo:disjoint-type-context}
Let $f$ be a disjoint-type function that is of finite order or can be written as a finite composition of finite-order functions in the class $\mathcal{B}$.
Then $J(f)$ is a Cantor bouquet, and $\meanderingendpoints(f) = J(f) \setminus A(f)$.
\end{thmx}

We then have the following.

\begin{theorem}
\label{theo:disjoint-type}
Let $f$ be a disjoint-type function that is of finite order or can be written as a finite composition of finite order functions in the class $\mathcal{B}$. Then $\escapingendpoints(f)~\cup~\{\infty\}$ is connected, but $\meanderingendpoints(f) \cup \{\infty\}$ is totally separated.
\end{theorem}
Note that the first conclusion is part of \cite[Theorem 1.9]{LasseNada}, and is included here for emphasis. Note also that every non-escaping endpoint belongs to $\meanderingendpoints(f)$. Hence \eqref{lbl:nonfastescaping} is an immediate consequence of Theorem~\ref{theo:disjoint-type}. This observation applies to many of our results (and also to the results of \cite{LasseVasso}). \\


We next consider functions that may be of infinite order, whose \textit{tracts} satisfy various simple geometric conditions; precise definitions of these conditions are given in Section \ref{log transform_sec}. Functions that satisfy these conditions were studied in \cite{DevHairs}. The setting for our result is given by the following.  
\begin{thmx}
\label{theo:disjoint-type-inf-order-context}
Suppose that $f$ is of disjoint type, and the tracts of a transform $F$ of $f$ have bounded slope and bounded wiggling. Then $J(f)$ is a Cantor bouquet, each component of which, apart possibly from the finite endpoint, lies in $I(f)$. If, in addition, $f$ has one tract, and the tracts of $F$ have bounded gulfs, then $\meanderingendpoints(f)~=~J(f)~\setminus~A(f)$.
\end{thmx}
Note that, unlike Theorem~\ref{theo:disjoint-type-context}, Theorem~\ref{theo:disjoint-type-inf-order-context} is not quite an immediate consequence of existing results, and so we give a brief proof of this in Section~\ref{infinite_sec}. Our result in the setting of Theorem~\ref{theo:disjoint-type-inf-order-context} is the following.
\begin{theorem}
\label{theo:disjoint-type-inf-order}
Suppose that $f$ is of disjoint type, and the tracts of a transform of $f$ have bounded slope and bounded wiggling. Then the dichotomy \eqref{lbl:escaping}/\eqref{lbl:nonfastescaping} holds. If, in addition, $f$ has one tract, and the tracts of a transform of $f$ have bounded gulfs, then $\escapingendpoints(f) \cup \{\infty\}$ is connected but $\meanderingendpoints(f) \cup \{\infty\}$ is totally separated.
\end{theorem}
Note that the hypotheses of the theorem imply that $f$ satisfies a so-called \textit{head-start} condition. In both parts, the  conclusion about escaping endpoints follows immediately from \cite[Theorem 1.9]{LasseNada} and, once again, is included for emphasis. \\


The first result on the topology of the set of non-escaping endpoints was given in \cite{Fweb}, and concerns \emph{Fatou's function} $f(z) := z+1+e^{-z}$, which, in fact, does not lie in the class $\mathcal{B}$. The dichotomy \eqref{lbl:escaping}/\eqref{lbl:nonfastescaping}, for this function, was studied further in \cite{LasseVasso}. Motivated by this result we study a large class of functions which occur as \textit{lifts} of a disjoint-type function; this class includes Fatou's function. 
The following can be deduced from \cite[Theorem 2]{semiconjugation} and a simple topological argument, together with \cite[Theorem 4.1]{DevHairs}; the proof is omitted.
\begin{thmx}
\label{theo:Fatou-functions-context}
Let $f$ be a transcendental entire function such that $\pi \circ f = g \circ \pi$, where $\pi(z) := \exp(az)$, for some $a \neq 0,$ and $g$ satisfies the hypotheses of Theorem \ref{theo:disjoint-type-context}. If $0$ is an attracting fixed point for $g$, then $J(f)$ is a Cantor bouquet, and $\meanderingendpoints(f) = J(f) \setminus A(f)$.
\end{thmx}
We are able to extend our result on the meandering endpoints to this class of functions outside the class $\mathcal{B}.$ In particular: 
\begin{theorem}
\label{theo:Fatou-functions}
Let $f$ be a transcendental entire function such that $\pi \circ f = g \circ \pi$, where $\pi(z) := \exp(az)$, for some $a \neq 0,$ and $g$ satisfies the hypotheses of Theorem \ref{theo:disjoint-type-context}. If $0$ is an attracting fixed point for $g$, then $\meanderingendpoints(f) \cup \{\infty\}$ is totally separated.
\end{theorem}

\begin{remark}
\emph{Note that, contrary to the previous cases, Theorem \ref{theo:Fatou-functions} does not include a result on the escaping endpoints. In particular, it is not known whether the union of escaping endpoints with infinity is a connected set. Note that it is not necessarily the case that $I(g) = \pi(I(f))$;  see, for example, \cite{david}.}
\end{remark}

Next, we consider the larger class of hyperbolic functions. We need a different notion of endpoints in this setting, since the Julia set may not be a Cantor bouquet. Suppose that $f \in \B$ is a hyperbolic function of finite order, or is a finite composition of such functions. It follows \cite[Corollary 5.3]{Rigidity} that $J(f)$ is a \emph{pinched Cantor bouquet}; in other words, the quotient of a Cantor bouquet by a closed equivalence relation on its endpoints. It follows that our definition of the set $\endpoints(f)$ -- together with its various subsets -- carries over into this setting in an obvious way. Note \cite[Theorem 1.2]{DevHairs} that if $z \in J(f) \setminus A(f)$, then $z$ is an endpoint.

It is natural to ask if Theorem~\ref{theo:disjoint-type} can be generalised from disjoint-type functions to this larger class. In fact this is impossible, because there are hyperbolic functions of finite order with bounded immediate attracting basins. An example, see \cite{BFR}, and, in particular, \cite[Figure 2(c)]{BFR}, is the function 
\[
f(z) := \frac{4\pi}{3}(1 - \cos z).
\]
Other examples of this behaviour occur in the family
\begin{equation}
\label{gdef}
g_\lambda(z) := \lambda z^2 \exp(z - z^2), \qfor \lambda > 0.
\end{equation}
We discuss this family of functions at the end of this paper.


It is, therefore, not the case that $\meanderingendpoints \cup \{\infty\}$ is necessarily totally separated for hyperbolic functions of finite order. However, we can obtain a result similar to Theorem~\ref{theo:disjoint-type} if we consider instead the set of endpoints which both escape and meander; this is defined by $\escapingandmeanderingendpoints(f) := \escapingendpoints(f) \setminus A(f)$. 
\begin{theorem}
\label{theo:hyperbolic-esc}
Suppose that $f$ is a hyperbolic function that is of finite order or can be written as a finite composition of finite order functions in $\mathcal{B}.$ Then $\escapingendpoints(f) \cup \{\infty\}$ is connected, but $\escapingandmeanderingendpoints(f) \cup \{\infty\}$ is totally separated.
\end{theorem}
Once again the first conclusion is \cite[Remark 7.2]{LasseNada}, and is included here for emphasis. Escaping endpoints as we define them are called \emph{escaping endpoints in the strong sense} in \cite[Remark 7.2]{LasseNada}. Note that  \cite[Remark 7.1]{LasseNada} gives a related but different definition of escaping endpoints for a class of functions containing functions considered in Theorem \ref{theo:hyperbolic-esc}. \\
%

Finally, our results have a connection with a topological structure called a spider's web, introduced by Rippon and Stallard \cite{Fast}. A set $E \subset \C$ is a spider's web if $E$ is connected, and there exists a sequence $(G_n)_{n\in\N}$ of bounded simply connected domains such that
\[
\partial G_n \subset E, \ G_n \subset G_{n+1}, \text{ for } n\in\N, \text{ and } \bigcup_{n\in\N} G_n = \mathbb{C}.
\]

We have the following extension of Theorem~\ref{theo:disjoint-type}.
\begin{theorem}
\label{theo:disjoint-type-spw}
Let $f$ be a disjoint-type function that is of finite order or can be written as a finite composition of finite order functions in the class $\mathcal{B}$. Then $F(f) \cup A(f)$ is a spider's web.
\end{theorem}
For reasons of brevity, we prove only this theorem but in fact, other, similar, results are possible. In particular, we note the following.
\begin{remark}\normalfont
We can show that $F(f)\cup A(f)$ is a spider's web for functions considered in Theorem \ref{theo:Fatou-functions}, thus generalising the first part of \cite[Theorem 5.1]{LasseVasso}. Since for these functions the Fatou set consists of one component which is a Baker domain, we deduce that $I(f)$ contains a spider's web, and hence, by \cite[Lemma 4.5]{rippon-stallard13}, $I(f)$ is a spider's web. This result generalises \cite[Remark 1.1 (a)]{Fweb}, where a subclass of functions satisfying the hypotheses of Theorem \ref{theo:Fatou-functions} was considered.
\end{remark}
\subsection*{Structure}
The structure of this paper is as follows. 
First, in Section \ref{log transform_sec}, we present some useful definitions and preliminary results. 
Next, in Section \ref{qc conjugacy_sec}, we prove some results regarding quasiconformally equivalent functions. 
We prove Theorem \ref{theo:disjoint-type} in Section \ref{main proof_sec}.
In Section \ref{infinite_sec} we consider functions which may be of infinite order, and we prove Theorem \ref{theo:disjoint-type-inf-order}. 
We give the proof of Theorem \ref{theo:Fatou-functions} in Section \ref{not B_sec}.
Section \ref{hyperbolic_sec} concerns hyperbolic functions and contains the proof of Theorem \ref{theo:hyperbolic-esc}.
Next, in Section~\ref{spw_sec} we discuss spiders' webs, and we prove Theorem~\ref{theo:disjoint-type-spw}.
Finally, in Section~\ref{examples_sec} we give detailed information about the dynamics of the family of functions in \eqref{gdef}.
\subsection*{Notation}
If $r \in \R$, then we set $r^+ := \max\{r, 0\}$. We denote the unit disc by $\D := \{ z \in \C : |z| < 1\}$ and the right half-plane by $\H := \{ z \in \C: \operatorname{Re} z > 0 \}$. If $S \subset \C$, then we denote the complement of $S$ in $\C$ by $S^c := \C \setminus S$.
%
%
%
%
\section{Properties of functions in the Eremenko-Lyubich class}
\label{log transform_sec}
In this section we give a number of important definitions, together with several useful results on logarithmic transforms of functions of disjoint type. Many of the ideas in this section are well-known, and we refer to, for example, \cite{EandL, Rigidity, RRRS} and the survey \cite{DaveSurvey} for more detail.
\subsection{The logarithmic transform and the class $\Blog$}
Suppose that $f\in\B$. Following the notation used in \cite{RRRS}, let $D$ be a bounded Jordan domain such that $S(f) \cup \{ 0, f(0) \} \subset D$. Let $W$ be the complement of the closure of $D$. Then the components of $V := \{z: f(z) \in W\}$ are simply connected Jordan domains called \textit{logarithmic tracts} of $f$. Note that $f|_V$ is a universal covering of $W$. 

If we take $\mathcal{T}:= \exp^{-1}(V)$ and $H:= \exp^{-1}(W)$, then there is an analytic function $F: \mathcal{T} \to H$ (which can be taken to be $2 \pi i$ periodic) with the property that $\exp (F(z))= f(\exp(z))$, for $z \in \mathcal{T}$. We call $F$ a \textit{logarithmic transform} of $f$. If $\overline{\mathcal{T}} \subset H$, then we say $F$ is of \emph{disjoint type}. It is readily shown that if $F$ is a logarithmic transform of a function $f \in B$, then $f$ is of disjoint type if and only if the same is true of $F$.

It follows from the construction that the following all hold.
\begin{enumerate}[(i)]
\item $H$ is a $2\pi i$ periodic Jordan domain containing a right half plane.
\item Every component of $\mathcal{T}$ is an unbounded Jordan domain, with real parts bounded below and unbounded above.
\item The components of $\mathcal{T}$ have disjoint closures and accumulate only at infinity.
\item For every component $T$ of $\mathcal{T}$, $F : T \to H$ is a conformal isomorphism that extends continuously to the closure of $T$ in $\C$.
\item For every component $T$ of $\mathcal{T}$, $\exp|_T$ is injective.
\item $\mathcal{T}$ is invariant under translation by $2\pi i$.
\end{enumerate}
We denote by $\Blog$ the class of all functions $F : \mathcal{T} \to H$ such that $H, \mathcal{T}$ and $F$ satisfy these conditions whether or not they arise as a logarithmic transform of a {\tef}.
\subsection{Classes of logarithmic transforms and tracts}
We now focus on functions in $\Blog$ with certain additional properties. We use these classes of logarithmic transforms in our main results.

First, we say that a function $F \in \mathcal{B}_{\log}$ is of finite order if 
\[
\limsup \frac{\log \operatorname{Re} F(w)}{\operatorname{Re}w} < \infty \;\;\text{as}\;\operatorname{Re}w \to \infty\;\operatorname{in}\;\mathcal{T}.
\]
Clearly, if $F$ is a logarithmic transform of a function $f \in B$, then $f$ is of finite order if and only if the same is true of $F$.

Second, we say that $F \in \mathcal{B}_{\log}$ is \textit{normalised} if $H$ is the right half-plane, and $|F'(z)| \geq 2$, for all $z\in \mathcal{T}$.  We denote the class of normalised functions in $\B_{\log}$ by $\Blog^n$; see \cite{DevHairs}.

Our remaining three definitions correspond to certain geometric constraints on the tracts of $F$, illustrated in Figure~\ref{fbad}. Roughly speaking, a tract has bounded slope if it lies in a sector, bounded wiggling if it does not ``double back'' on itself too much, and bounded gulfs if any ``gulfs'' that reach forward, never reach too far forward. We now give the more precise definitions.
\begin{definition}
Let $F \in \B_{\log}$. A tract $T$ of $F$ has \emph{bounded slope} with constants $\alpha, \beta > 0$ if
$$\lvert \operatorname{Im} w_1 - \operatorname{Im} w_2 \rvert \leq \alpha \max\{\operatorname{Re} w_1, \operatorname{Re} w_2, 0\} + \beta,\;\;\text{ for all}\; w_1, w_2 \in T.$$
\end{definition}
If all the tracts of $F$ have bounded slope for the same constants, then we say that the tracts have \textit{uniformly bounded slope}. 

\begin{definition}
Let $F \in \mathcal{B}_{\log}$. A tract $T$ of $F$ has \emph{bounded wiggling} with constants $K'> 1$ and $\mu >0$ if, for each point $w_0\in \overline{T}$, every point $w$ on the hyperbolic geodesic of $T$ that connects $w_0$ to $\infty$ satisfies
$$(\operatorname{Re}w)^+ > \frac{1}{K'} \operatorname{Re}w_0 - \mu.$$
\end{definition}
If all the tracts of $F$ have bounded wiggling for the same constants, then we say that the tracts have \emph{uniformly bounded wiggling}.
It is an important fact, see \cite[Theorem 5.6]{RRRS}, that if $F$ is a function of finite order in the class $\mathcal{B}_{\log}^n$ (or even a finite composition of such functions), then the tracts of $F$ have uniformly bounded slope and uniformly bounded wiggling.

\begin{definition}
Let $F \in \Blog^n,$ let $T$ be a tract of $F$, and let $p$ denote the point in $\partial T$ for which $F(p) = 0$.
\begin{enumerate}[(a)]
\item If $w \in \overline{T}$ and $a > \operatorname{Re} w$, then $L_{w, a}$ denotes the unique component of the set $\{ w' : \operatorname{Re} w' = a \} \cap T$ that separates $w$ from $\infty$ in $T$.
\item The tract $T$ has \emph{bounded gulfs} with constant $C > 1$ if $L_{w,a}$ separates $p$ from infinity,
for all $w \in T,$ $a > 0,$ with $\operatorname{Re} w \geq \max\{\operatorname{Re} p, 1\}$ and $a \geq C\operatorname{Re} w.$
\end{enumerate}
\end{definition}

\begin{figure}
	\includegraphics[width=14cm,height=8cm]{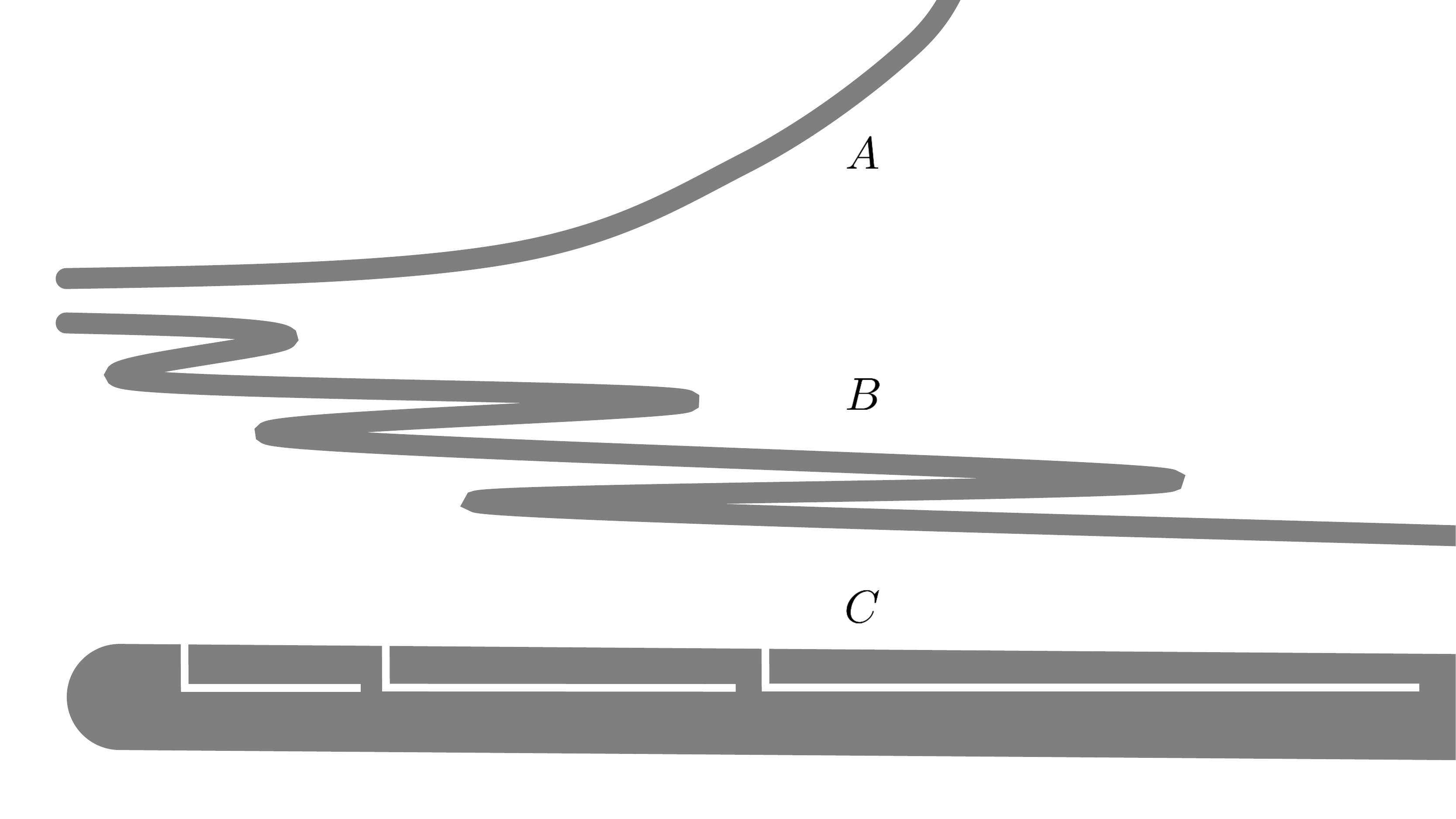}
  \caption{In this approximate graphic, in which the tracts are in gray, tract $A$ does not have bounded slope,  tract $B$ does not have bounded wiggling, and tract $C$ does not have bounded gulfs.}\label{fbad}
\end{figure}
\subsection{Results in tracts}
We now state three results that we need, which all study the behaviour of the logarithmic transform in tracts that have certain properties. The first is \cite[Lemma 3.2]{DevHairs}.
\begin{lemma}
\label{growlemma}
Let $F \in \Blog^n$, and let $K > 1$ and $\alpha, \beta >0$. Let $T$ be a tract of $F$ with bounded wiggling, with constants $K' > 1$ and $\mu
> 0$. Then there exist constants $\epsilon = \epsilon(K) > 0$ and $M = M(K, K', \mu, \alpha, \beta) > 0$ such that if $w, z \in \overline{T}$,
\[
\operatorname{Re} w > K (\operatorname{Re} z)^+ + M,
\]
and
\begin{equation}
\label{e1}
|\operatorname{Im} F(w) - \operatorname{Im} F(z)| \leq \alpha \max\{\operatorname{Re} F(w), \operatorname{Re} F(z)\} + \beta,
\end{equation}
then
\[
\operatorname{Re} F(w) > \exp(\epsilon \operatorname{Re} w) \operatorname{Re} F(z).
\]
\end{lemma} 
The second result we need is \cite[Lemma 5.3]{DevHairs}.
\begin{lemma}
\label{inflemma}
Let $F \in \Blog^n$, let $T$ be a tract of $F$, and let $p$ be the point of $\partial T$ for which $F(p) = 0$. Suppose that $T$ has bounded gulfs with constant $C>1$, and bounded wiggling with constants $K' > 1$ and $\mu > 0$. Then there is a constant $D=D(C,K',\mu)>1$ such that if $A > \max\{\operatorname{Re} p, 1\}$ and $a \geq DA$, then the following both hold.
\begin{enumerate}[(a)]
\item $L_{z, a} = L_{p, a}$, for $z \in T$ and Re $z = A$.\label{inflemmaa}
\item $\max_{w \in L_{p, a}} \operatorname{Re} F(w) \geq \max_{w \in T, \operatorname{Re} w = A} \operatorname{Re} F(w)$.\label{inflemmab}
\end{enumerate}
\end{lemma} 
Finally, we use the following \cite[Lemma 5.4]{DevHairs}.
\begin{lemma}
\label{inflemma2}
Let $F \in \Blog^n$, let $T$ be a tract of $F$, and let $p$ be the point of $\partial T$ for which $F(p) = 0$. If $a > \operatorname{Re} p$, and $b > a + 4 \pi$, then
\[
\frac{|F(w_b)|}{|F(w_a)|} \geq \exp\left(\frac{1}{2}(b-a) - 4\pi\right), \qfor w_a \in L_{p, a}, w_b \in L_{p, b}.
\]
\end{lemma} 
\subsection{Dynamically important sets of a function in class the $\Blog$}
In this subsection we define the Julia, escaping and fast escaping sets of a function $F$ in the class $\Blog$. The Julia set is defined by
\[
J(F) := \{ z \in \overline{\mathcal{T}} : F^n(z) \in \overline{\mathcal{T}}, \text{ for } n \in \N\}.
\]

The escaping set is a subset of the Julia set, defined by
\[
I(F) := \{ z \in J(F) : \operatorname{Re } F^n(z) \rightarrow\infty \text{ as } n \rightarrow\infty\}.
\]

The fast escaping set is a subset of the escaping set, defined by 
\[
A(F) := \{z \in J(F) : \text{there exists } \ell \in \N \text{ s.t. } \operatorname{Re} F^{n+\ell}(z) \geq M^n(R,F), \text{ for } n \in \N\},
\]
where
\[
M(r,F) := \max_{\operatorname{Re} w = r} \operatorname{Re} F(w),
\]
and $R>0$ is any value so large that $M(r,F) > r$ for $r \geq R$.

%
\subsection{The logarithmic transform of a disjoint-type function}
In the remainder of this section we suppose that $F \in \Blog$ is a logarithmic transform of a disjoint-type function $f$. It can be seen that, with this assumption, we have that $J(f) = \exp J(F)$. Since the exponential is locally a homeomorphism, it can be deduced from the definition that $J(f)$ is a Cantor bouquet if and only if $J(F)$ is one. We use this fact in several arguments. 

Each branch of the logarithm maps a component of $J(f)$ to a component of $J(F)$. Hence, we can use $\endpoints(F)$, $\escapingendpoints(F)$ and $\meanderingendpoints(F)$ to refer to the endpoints, non-escaping endpoints and meandering endpoints of $F$. Note that for a disjoint-type function we have 
\[
I(f) = \exp I(F), \quad\text{ and }\quad A(f) = \exp A(F).
\]
It follows that we also have 
\[
\meanderingendpoints(f) = \exp \meanderingendpoints(F) \quad\text{ and }\quad \escapingendpoints(f) = \exp \escapingendpoints(F).
\]

For completeness, we briefly sketch a proof that $A(f) = \exp A(F)$. Choose $R>0$ large. Suppose that $z = \exp w$. We have that $w \in A(F)$ if and only if the forward orbit of $w$ lies in the tracts of $F$, and also there exists $\ell \in \N$ such that $\operatorname{Re} F^{n+\ell}(w) \geq M^n(R, F)$, for $n \in \N$. Since the complement of the tracts of $f$ lies in $F(f)$, we have that $z \in A(f)$ if and only if the forward orbit of $z$ lies in the  tracts of $f$, and also there exists $\ell \in \N$ such that $|f^{n+\ell}(z))| \geq M^n(e^R, f)$, for $n \in \N$. Recall that $\exp \circ F = f \circ \exp$. We deduce that $\exp \operatorname{Re} F^{n+\ell}(w) = |f^{n+\ell}(z)|$ and $\exp M^n(R, F) = M^n(e^R, f)$. The result follows.
\subsection{Other results and definitions}
In order to prove our result on meandering endpoints we make use of the following 
\cite[Lemma 2.8]{LasseVasso}. Recall that points $a, b \in X$ are separated in a metric space $X$ if there is an open and closed set $U \subset X$ that contains $a$ but not $b$.
\begin{lemma}
\label{lemm:LasseVasso}
Suppose that $X$ is a metric space, with $x \in X$, and that $A := X\setminus\{x\}$ is totally separated. If, in addition, every point of $A$ is separated from $x$ in $X$, then $X$ is totally separated.
\end{lemma}

We also use the following definition.
\begin{definition}\label{def:set-sep}
 If $x,y\in \widehat{\C}$, we say that $E\subset \widehat{\C}$ \emph{separates} $x$ from $y$ 
      if $x$ and $y$ are separated in
      $(\C\setminus E)\cup \{x,y\}$. 
      \end{definition}
			
Finally we use the following, which is part of \cite[Lemma 1]{Slow}.
\begin{lemma}
\label{RSlemma}
Suppose that $(E_n)_{n \geq 0}$ is a sequence of compact sets in $\C$, and that $f : \C \to \C$ is a continuous function such that
\[
f(E_n) \supset E_{n+1}, \qfor n \geq 0.
\]
Then there exists $\zeta$ such that $f^n(\zeta) \in E_n$, for $n \geq 0$.
\end{lemma}
%
%
%
\section{Functions related by quasiconformal maps}
\label{qc conjugacy_sec}
In many of our proofs we need to transfer certain properties of one function to a second function that is related to the first by some quasiconformal map. In this section we give some results which are required to achieve this.

We say that two functions $F,G  \in \B_{\log}$, with domains $V$ and $W$, are \textit{quasiconformally equivalent} if there are quasiconformal maps  $\Phi, \Psi: \C \to \C$ with the following four properties; here for $R\in\R$ we use the notation $\H_R := \{ z : \operatorname{Re} z > R \}$.
\begin{enumerate}[(i)]
\item The maps $\Phi$ and $\Psi$ each commute with $z \mapsto z+2\pi i$.
\item We have that $\operatorname{Re}\Phi(z) \to \pm \infty$ as $\operatorname{Re}z \to \pm\infty$ (and the same holds for $\Psi$).
\item For sufficiently large $R$, $\Phi(F^{-1}(\mathbb{H}_R)) \subset W$ and $\Phi^{-1}(G^{-1}(\mathbb{H}_R)) \subset V$. 
\item $\Psi \circ F = G \circ \Phi$, wherever both compositions are defined.
\end{enumerate}

We need the following, which is part of \cite[Theorem 3.1]{Rigidity}. Note that the final statement is not given in \cite{Rigidity}, but follows easily from the proof.
\begin{theorem}
\label{lasseconj}
Suppose that two disjoint-type functions in $\Blog$ 
\[
F : V \to H, \quad\text{and}\quad G : \Phi(V) \to \Psi(H),
\]
are quasiconformally equivalent. Then there is a quasiconformal map $\Theta : \C \to \C$ such that $\Theta \circ F = G \circ \Theta$ on $V$,
\begin{equation}
\label{thetaeq}
\Theta(z + 2\pi i) = \Theta(z) + 2\pi i, \qfor z \in \C,
\end{equation}
and $\operatorname{Re} \Theta(z) \rightarrow -\infty$ as $\operatorname{Re} z \rightarrow -\infty$. 
\end{theorem}
%
%
We use Theorem~\ref{lasseconj} to deduce the following.
\begin{proposition}
\label{prop:thereisaconjugacy}
Suppose that $f, g \in \B$ are of disjoint type. Suppose also that $F$ (resp. $G$) are logarithmic transforms of $f$ (resp. $g$) that are quasiconformally equivalent. Then there is a quasiconformal map $\vartheta : \C \to \C$ and a domain $U$ with $J(f) \subset U$ and $J(g) \subset \vartheta(U)$, such that $\vartheta \circ f = g \circ \vartheta$ in $U$.
\end{proposition}
\begin{proof}
Let $\Theta$ be the quasiconformal map from Theorem~\ref{lasseconj}. Set 
\[\vartheta(z) = 
\begin{cases}
(\exp \circ \ \Theta \circ \log)(z), &\text{for } z \ne 0 \\
0, &\text{for } z = 0.
\end{cases}
\]
This is well-defined by \eqref{thetaeq}. Set $U = \exp V$, where $V$ is the domain of $F$. It can be seen that $\vartheta$ has the necessary properties.
\end{proof}
It is useful for later in the paper to define another subset of the escaping set of a {\tef}. For a {\tef} $f$, and $R>0$, we define
\begin{equation}
\label{Xdef}
X(f) := \{z \in\C : \exists \ \ell \in \N \text{ such that } |f^{n+\ell}(z)| \geq \exp^n (R) \text{ for } n\in\N\}.
\end{equation}
It can be shown that this definition is independent of the choice of $R>0$. A set similar to $X(f)$ (denoted by $B(f)$) was used in \cite[Proof of Theorem 1.2]{bergweiler2017}. For some functions, including all functions of finite order, we have $X(f)=A(f)$. This is not the case in general; unlike $A(f)$, the set $X(f)$ can be empty. It is not hard to see that $X(f) \neq \emptyset$, for $f \in \B$.

We require the following, which is slightly more general than is required, and may be of independent interest.
\begin{theorem}
\label{theo:implicationsofaconjugacy}
Suppose that $f, g$ are {\tef}s, and that the map $\vartheta : \C \to \C$ is a homeomorphism. Suppose that $\vartheta \circ f = g \circ \vartheta$ in a domain $U$ such that $J(f) \cup I(f) \subset U$ and $J(g) \cup I(g) \subset \vartheta(U)$. Then 
\[
\vartheta(J(f)) = J(g), \quad \vartheta(I(f))=I(g), \quad\text{ and }\quad \vartheta(A(f)) = A(g).
\]
If, in addition, $\vartheta$ is quasiconformal, then $\vartheta(X(f))=X(g)$. 
\end{theorem}
\begin{remark}\normalfont
Even with $U = \C$, this result does not appear to have been used in the literature.
\end{remark}
\begin{proof}[Proof of Theorem~\ref{theo:implicationsofaconjugacy}]
First, suppose that $\vartheta$ is a homeomorphism. The facts that $\vartheta(J(f)) = J(g)$ and $\vartheta(I(f))=I(g)$ are easily shown.

To show that $\vartheta(A(f)) = A(g)$, we use the fact from \cite[Corollary 2.7]{Fast} that 
\begin{equation}
\label{newAdef}
A(f) := \{w \in \C : \exists \ \ell \in \N \text{ such that } f^{n+\ell}(w) \notin T(f^n(D)), \text{ for } n \in \N\},
\end{equation}
where $D$ is a domain that meets $J(f)$. Here, if $V \subset \C$, we use $T(V)$ to denote the union of $V$ with all bounded components of the complement of $V$.

Because of the symmetry of the hypotheses, it is only necessary to prove that $\vartheta(A(f)) \subset A(g)$. Choose $w \in A(f)$, and let $z = \vartheta(w)$. Let $D \subset U$ be a domain that meets $J(f)$. Since $\vartheta(J(f)) = J(g)$, it follows that $D' := \vartheta(D)$ meets $J(g)$. 

Then, by \eqref{newAdef}, there exists $\ell \in \N$ such that 
\begin{equation*}
f^{n+\ell}(w) \in U \setminus T(f^n(D)), \qfor n \in \N. 
\end{equation*}
We deduce that
\begin{equation}
\label{finA}
\vartheta^{-1}(g^{n+\ell}(z)) \in U \setminus T(\vartheta^{-1}(g^n(D'))), \qfor n \in \N. 
\end{equation}

It can be seen that if $\psi : \C \to \C$ is a homeomorphism, and $V \subset \C$ is a domain, then $\psi(T(V)) = T(\psi(V))$; this follows immediately from the fact that $\psi$ induces a bijection between the set of bounded complementary components of $V$ and the set of bounded complementary components of $\psi(V)$.

It then follows from \eqref{finA} that
\begin{equation*}
g^{n+\ell}(z) \in \vartheta(U) \setminus T(g^n(D')), \qfor n \in \N. 
\end{equation*}
We can deduce from this that $z \in A(g)$ as required. 

For the final part of the proposition, suppose that $\vartheta$ is quasiconformal. It follows from the fact that $\vartheta$ is quasiconformal, and so H\"{o}lder continuous, that there exist $R_0>0$ and $s>0$ such that
\begin{equation}
\label{growtheq}
|\vartheta(z)| > |z|^s, \qfor |z| > R_0.
\end{equation}

Because of the symmetry of the hypotheses, it is only necessary to prove that $\vartheta(X(f)) \subset X(g)$. Choose $w \in X(f)$, and let $z = \vartheta(w)$. Choose $R>R_0$, and also $R' > R_0$ sufficiently large that $(\exp^n(R'))^s \geq \exp^n(R)$, for $n \in \N$.

Then, by definition, there exists $\ell \in \N$ such that 
\begin{equation}
\label{finAexp}
|f^{n+\ell}(w)| \geq \exp^n(R'), \qfor n \in \N. 
\end{equation}
It then follows by \eqref{growtheq} and \eqref{finAexp} that
\begin{align*}
|g^{n+\ell}(z)|    &=    |\vartheta(f^{n+\ell}(w))|, \\
                   &\geq |f^{n+\ell}(w)|^s, \\
                   &\geq (\exp^{n}(R'))^s, \\
							     &\geq \exp^n(R).
\end{align*}

It is easy to deduce from this that $z \in X(g)$ as required.
\end{proof}
%
%
%
\section{Proof of Theorem~\ref{theo:disjoint-type}}
\label{main proof_sec}
In this section we prove Theorem \ref{theo:disjoint-type}. We only consider the case where $f \in \B$ is of disjoint type and finite order. The proof in the case that $f$ is a composition of such functions uses techniques similar to those used in \cite{DevHairs}, and is omitted for reasons of simplicity.

Suppose then that $f \in \B$ is of disjoint type and finite order. Let $F$ be a logarithmic transform of $f$. We begin by fixing a function in $\Blog^n$ which is quasiconformally equivalent to $F$. To do this, first choose
$L > 1$ sufficiently large so that $g(z) := f(z)/L$ is of disjoint type, and also is such that  
\begin{equation}
\label{niceeq}
S(g) \cup \{g(0)\} \subset \D \subset \{ z : |z| \leq e^2 \} \subset F(g).
\end{equation}

Let $G$ be the logarithmic transform of $g$ which is a conformal map from each component of $\mathcal{T} := G^{-1}(\H)$ to $\H$. It follows from  \cite[Lemma 1]{EandL} that there exists $R_0>0$ such that $|G'(z)| \geq 2$, for all $z\in \mathcal{T}$ such that $\operatorname{Re}G(z) \geq R_0$. Replacing $g$ with the function $z \to e^{-R_0}g(z)$ is equivalent to replacing $G$ with the function $z \to G(z) - R_0$. It follows that, by choosing a larger value of $L$ if necessary, we can assume both that \eqref{niceeq} holds, and that $G$ is normalised.

The following lemma is central to the proof of our theorem.
\begin{lemma}
\label{lemma:Gisnice}
Let $G \in \Blog^n$ be as above. Then every point in $\meanderingendpoints(G)$ can be separated from infinity by a continuum $\gamma \subset \meanderingendpoints(G)^c = A(G) \cup J(G)^c$.
\end{lemma}
\begin{remarks}\normalfont
\mbox{ }
\begin{enumerate}[(a)]
\item Recall that the fact that $\gamma \subset A(G) \cup J(G)^c$ separates a point $z \in \meanderingendpoints(G)$ from infinity means that $z$ and infinity are separated in $\meanderingendpoints(G) \cup \{\infty\}$ (see Definition~\ref{def:set-sep}).

\item In fact, the stronger conclusion that every point in $\C$ can be separated from infinity by a continuum in $A_R(G) \cup J(G)^c$ can be proved, where $A_R(G)$ is a so-called \emph{level} set; see \cite{Fast}. We do not require this additional detail.
\end{enumerate}
\end{remarks}
\begin{proof}[Proof of Lemma~\ref{lemma:Gisnice}]
The fact that $\meanderingendpoints(G)^c = A(G) \cup J(G)^c$ is a consequence of the fact that $G$ is of disjoint-type, together with Theorem~\ref{theo:disjoint-type-context} (applied to $g$).

Since the tracts of $G$ have uniformly bounded slope, equation \eqref{e1} holds, for some $\alpha, \beta>0$ which depend only on $G$, whenever the points $G(w)$ and $G(z)$ lie in the same tract of $G$. Let $\epsilon>0$ and $M>0$ be the constants from Lemma~\ref{growlemma} for these values of $\alpha$ and $\beta$, and with $K=2$. We also define the function
\begin{equation}
\label{taudef}
\tau(r) := \exp(\epsilon r).
\end{equation}

Let $T_0, T_1, \ldots$ be the tracts of $G$. 
%
Choose 
\[
R' > \max\left\{\frac{\log 10}{9\epsilon}, 4M\right\}.
\]

Since $G$ is of finite order, it is easy to see that there exist $\delta>0$ and $r_0>0$ such that
\[
r < M(r, G) \leq e^{\delta r}, \qfor r \geq r_0.
\]
It then follows from \cite[Lemma 3.4]{DevHairs} that, increasing $R'$ if necessary, we can assume that
\begin{equation}
\label{bigenough}
\tau^n(R') \geq M^n(r_0, G), \qfor n\in \N.
\end{equation}

For simplicity, set $G(z) = 0$, for $z \notin \mathcal{T}$, so that $G(\C\setminus\mathcal{T}) \subset \C\setminus\mathcal{T}$. Define
\[
B := \bigcap_{k \geq 0} \{ z \in \C : \operatorname{Re} G^k(z) \geq \tau^k(R') \text{ or } G^k(z) \notin \mathcal{T} \}.
\]
%
Note that $B$ is closed, and that $B^c \subset \mathcal{T}$. 

We shall prove that all components of the complement of $B$ are bounded. Suppose, by way of contradiction, that $B^c$ has an unbounded component, say $\Gamma$. Note that $\Gamma \subset \mathcal{T}$, and so there is a tract $T_0$ of $G$ such that $\Gamma \subset T_0$.
 
Since $\Gamma$ is unbounded and open, there exist points $z_0', z_0'', z_0''' \in T_0$ and a curve $\gamma_0 \subset \Gamma$ joining $z_0', z_0''$ and $z_0'''$ such that the following all hold. 
\begin{itemize}
\item $\operatorname{Re} z_0' = a_0 > R'$.
\item $\operatorname{Re} z_0'' = 3a_0 + M$.
\item $\operatorname{Re} z_0''' = 9a_0 + 4M$.
\item $\gamma_0 \subset \{ z \in T_0 : \operatorname{Re} z \in [a_0, 9a_0 + 4M]\}$.
\end{itemize}

We next use induction to prove that there is a sequence of curves $(\gamma_k)_{k\geq 0}$ such that the following all hold, for $k \geq 0$.
\begin{enumerate}[(a)]
\item $\gamma_{k+1} \subset G(\gamma_k)$. \label{propa}
\item $\gamma_k$ joins points  $z_k', z_k''$ and $z_k'''$. \label{propb}
\item $\operatorname{Re} z_k' = a_k > \tau^k(R')$. \label{propc}
\item $\operatorname{Re} z_k'' = 3a_k + M$. \label{propd}
\item $\operatorname{Re} z_k''' = 9a_k + 4M$. \label{prope}
\item There is a tract $T_k$ such that $\gamma_k \subset \{ z \in T_k : \operatorname{Re} z \in [a_k, 9a_k + 4M]\}$. \label{propf}
\end{enumerate}

We need to prove that if this claim holds for all $0 \leq j \leq k$, then it also holds for $j = k+1$. Assume that curves with properties \eqref{propa}--\eqref{propf} have been constructed for $0 \leq j \leq k$. Note that it follows from \eqref{propa} and \eqref{propc} that if $\gamma'_k$ is the component of $G^{-k}(\gamma_k)$ contained in $\gamma_0$, then $\operatorname{Re} G^n(z) > \tau^n(R')$, for $z \in \gamma'_k$ and $0 \leq n \leq k$. Since $\gamma_0 \cap B = \emptyset$, it follows by \eqref{propc} that $G(\gamma_k) \subset \mathcal{T}$ and so, in particular, there is a tract $T_{k+1}$ such that $G(\gamma_k) \subset T_{k+1}$. By \eqref{niceeq} we have $\operatorname{Re} G(z)  > 2$, for $z \in \gamma_k$.

It follows from Lemma~\ref{growlemma}, with $K=2$ and $z = z_k'$, that 
\[
\operatorname{Re} G(w) > 2 \exp(\epsilon \operatorname{Re} w) > \tau(\operatorname{Re} w ), \qfor w \in \gamma_k \cap \{ z : \operatorname{Re} z \geq 3a_k + M\}.
\]

A second application of Lemma~\ref{growlemma}, with $K=2$, $w = z_k'''$, and $z = z_k''$, gives, by \eqref{propc} and the choice of $R'$, that
\begin{align*}
\operatorname{Re} G(z_k''') &> \exp(\epsilon(9a_k + 4M)) \operatorname{Re} G(z_k'') \\
                             &> 10 \operatorname{Re} G(z_k'') \\
														 &> 9 \operatorname{Re} G(z_k'') + 4M.
\end{align*}

Our claim then follows immediately. \\

By Lemma~\ref{RSlemma}, there exists $\zeta \in \gamma_0$ such that $\operatorname{Re} G^n(\zeta) \geq \tau^n(R')$, for $n \in \N$. This implies that $\zeta \in B$, which is a contradiction. This completes the proof that all components of the complement of $B$ are bounded. \\

We next claim that $B \subset A(G) \cup J(G)^c$. 
For, if $z \in B \cap J(G)$, then it follows from \eqref{bigenough}, together with the definition of $B$, that 
\[
\operatorname{Re} G^n(z) \geq \tau^n(R') \geq M^n(r_0, G), \qfor n \in \N,
\]
and so $z \in A(G)$ as required. \\

Finally, suppose that $z \in \meanderingendpoints(G)$. Since $\meanderingendpoints(G) = J(G) \setminus A(G) \subset B^c$, 
we can let $X$ be the component of $B^c$ containing $z$. We know that this set is bounded. Since $B$ is closed, it follows that $\partial X$ is a continuum in $B \subset A(G) \cup J(G)^c$ which separates $z$ from infinity, as required.
\end{proof}
\begin{proof}[Proof of Theorem~\ref{theo:disjoint-type} in the case that $f$ is a finite order function of disjoint type]
Let $F$ be a logarithmic transform of $f$, and let $G$ be as above. Then $F$ and $G$ are both of disjoint type, and are clearly quasiconformally equivalent. Let $\vartheta$ be the quasiconformal map that results from Proposition~\ref{prop:thereisaconjugacy}.

First we claim that every point in $\meanderingendpoints(f)$ can be separated from infinity by a continuum $\gamma \subset A(f) \cup F(f)$. For, suppose that $z \in \meanderingendpoints(f)$. It follows from Theorem~\ref{theo:implicationsofaconjugacy} that $\vartheta(z) \in \meanderingendpoints(g)$. Let $w$ be such that $\exp w = \vartheta(z)$ (recall that $0 = \vartheta(0) \in F(g)$). Since $G$ is of disjoint type, $w$ lies in $\meanderingendpoints(G)$. It follows by Lemma~\ref{lemma:Gisnice} that $w$ can be separated from infinity by a continuum $\gamma' \subset A(G) \cup J(G)^c$. Then $\gamma = \vartheta^{-1}(\exp \gamma')$ has the required properties. This completes the proof of the claim. 

It remains to prove that $\meanderingendpoints(f) \cup\{\infty\}$ is totally separated. Since, by Theorem~\ref{theo:disjoint-type-context}, $J(f)$ is Cantor bouquet,  the set $\endpoints(f)$ is totally separated. Moreover, it follows from the above that each point of $\meanderingendpoints(f)$ is separated from infinity in $\meanderingendpoints(f) \cup \{\infty\}$. The result then follows by Lemma~\ref{lemm:LasseVasso}.
%
\end{proof}
%
%
%
%
\section{Functions which may be of infinite order}
\label{infinite_sec}
This section concerns functions, which may be of infinite order, for which the tracts of a logarithmic transform have certain geometric properties that allow us to obtain similar results on the endpoints. We first show that the endpoints are defined in the same way as before, by proving Theorem~\ref{theo:disjoint-type-inf-order-context}.

\begin{proof}[Proof of Theorem~\ref{theo:disjoint-type-inf-order-context}]
Let $F$ be a logarithmic transform of $f$. For some $L > 0$, let $g(z) := f(z) / L$, and let $G$ be a logarithmic transform of $g$. By choosing $L$ sufficiently large, we can assume that $G$ is of disjoint type, is normalised, and its tracts have uniformly bounded slope and uniformly bounded wiggling. It can then be deduced from \cite[Corollary 6.3]{brushing} that $J(G)$ is a Cantor bouquet.  

Moreover, \cite[Theorem 4.7]{RRRS} gives that each connected component of $J(G)$ is a closed arc to infinity all of whose points except possibly the finite endpoint lie in the escaping set. Since $g$ is of disjoint type these properties also hold for $J(g)$. It follows by Proposition~\ref{prop:thereisaconjugacy} and Theorem~\ref{theo:implicationsofaconjugacy} that the same is true of $J(f)$. 

We now assume that $f$ has one tract and a logarithmic transform of f has tracts of bounded slope, bounded wiggling and
bounded gulfs, and so the same is true for $g$. It follows by \cite[Remark 3 after Theorem 5.2]{DevHairs} that every point $z \in I(G)$ can be connected to infinity by a curve $\gamma \in I(G)$ such that $\gamma\setminus\{z\} \subset A(G)$. 
We can deduce that each curve in $J(f)$, except possibly its endpoint, lies in the fast escaping set.
\end{proof}


In order to prove the first part of Theorem \ref{theo:disjoint-type-inf-order} we need the following result. Recall the definition of the set $X(f)$ in {\eqref{Xdef}
\begin{equation*}
X(f) := \{z \in\C : \text{ there exists } \ell \in \N \text{ such that } |f^{n+\ell}(z)| \geq \exp^n (R) \text{ for } n\in\N\}.
\end{equation*}
We have the following. 
\begin{lemma}
\label{lemm:otherinf}
Suppose that $f$ is a disjoint-type function, and the tracts of a transform of $f$ have uniformly bounded slope and uniformly bounded wiggling. Then $(\endpoints(f) \setminus X(f)) \cup \{\infty\}$ is totally separated.
\end{lemma}
\begin{proof}
This result is a consequence of the following observation. In the proofs of \cite[Theorem 1.2]{DevHairs} and Theorem \ref{theo:disjoint-type}, the assumption of finite order is used in two ways. The first is to deduce that the tracts of the function have uniformly bounded slope and uniformly bounded wiggling, and then to argue from this that certain points escape to infinity faster than an iterated exponential. The second is to deduce from this that these points are, in fact, in the fast escaping set. If we omit the assumption of finite order, and instead assume only that the tracts have uniformly bounded slope and wiggling, we still obtain the first of these two facts. 

In particular, it follows from Theorem \ref{theo:disjoint-type-inf-order-context} that $J(f)$ is a Cantor bouquet. By following the proof of \cite[Theorem 1.2]{DevHairs}, we can deduce that $$\endpoints(f) \setminus X(f) = J(f) \setminus X(f).$$ The rest of the proof is then very similar to the proof of Theorem \ref{theo:disjoint-type}, and is omitted.
\end{proof}

Suppose now that $f$ is of disjoint type and has only one tract. Suppose also that the tracts of a transform $F$ of $f$ have bounded slope, bounded wiggling and bounded gulfs. As in Section \ref{main proof_sec} we choose $L > 1$ sufficiently large so that $g(z) := f(z)/L$ is of disjoint type, and also equation \eqref{niceeq} holds.

Let $G$ be a logarithmic transform of $g$. 
As before, we can assume both that \eqref{niceeq} holds, and that $G$ is normalised, i.e. $G \in \Blog^n$. Note that since $f$ has only one tract, all the tracts of $G$ are $2\pi i$ translates of each other. We can assume, then, that the tracts of $G$ have uniformly bounded slope with constants $\alpha, \beta > 0$, uniformly bounded wiggling with constants $K' > 1$ and $\mu > 0$, and finally uniformly bounded gulfs with constant $C > 1$. Let $\epsilon, M$ be the constants from Lemma~\ref{growlemma}, for these values of $\alpha$ and $\beta$, and with $K=2$. Note also that 
\begin{equation}
\label{Mintract}
M(r, G) = \max_{w \in T_n, \operatorname{Re} w = r} \operatorname{Re} G(w), \qfor r> 0, \ n \geq 0.
\end{equation}

The following lemma is somewhat familiar.
\begin{lemma}
\label{inf-sep_lem}
Every point in $\meanderingendpoints(G)$ can be separated from infinity by a continuum $\gamma \subset A(G) \cup J(G)^c$.
\end{lemma}
\begin{proof}
Let $T_0, T_1, \ldots$ be the tracts of $G$, and let $\mathcal{T}$ denote the union of the tracts. Again we fix $G(z) = 0$, for $z \notin \mathcal{T}$.

We let $p_n$ be the point of $\partial T_n$ such that $G(p_n) = 0$, and note that these points are $2\pi i$ translates of each other. Set $p = \operatorname{Re} p_n$, for some $n \in \N$. Let $p' = \max\{p, 1\}$. Let $D > 1$ be the constant from Lemma~\ref{inflemma}, and set $\epsilon' = (4DK')^{-1} < 1$.

For simplicity we first prove the following.
\begin{proposition}
\label{prop:infnew}
Suppose that $T$ is a tract of $G$, $x_2 > \max\{2 \mu K', 2DK'p'\}$, and $x_1~\in~(p', \frac{x_2}{2DK'})$. If $\gamma \subset T$ is a curve which includes a point of real part $x_1$ and a point of real part $x_2$, then
\[
\gamma \cap L_{p, t} \ne \emptyset, \qfor t \in \left[Dx_1, \frac{x_2}{2K'}\right].
\]
\end{proposition} 
\begin{proof}
Suppose that $\gamma \subset T$ is a curve which includes a point $w$ of real part $x_1$ and a point $w'$ of real part $x_2$. Fix a value of $t \in \left[Dx_1, \frac{x_2}{2K'}\right]$. 

By Lemma~\ref{inflemma} \eqref{inflemmaa} we have that $L_{p, t} = L_{w, t}$. Thus $w$ lies in a bounded component of $T\setminus L_{p, t}$. By way of contradiction, suppose that $\gamma$ does not meet $L_{p, t}$. Then $w'$ also lies in the same bounded component of $T \setminus L_{p,t}$. Thus the geodesic joining $w'$ to $\infty$ in $T$ contains a point $w''$ with Re $w'' = t \leq x_2 / 2K'$. On the other hand, the hypothesis of bounded wiggling implies that $t > x_2/K' - \mu$. This is in contradiction to the choice of $x_2$.
\end{proof}

Choose 
\[
\alpha_1 > \max\left\{8\pi, 2Dp', \left(-4\log(\epsilon'(4(\alpha+2))^{-1}) + 16\pi\right)\right\},
\]
and
\[
\alpha_2 > \max\{3\alpha_1, 2DK'p', \mu K'\},
\]
and finally
\[
R' > \max\left\{p', \frac{1}{\epsilon}, M, 2Dp'\right\}.
\]
Increasing $R'$ if necessary, we assume that $M(\epsilon'r, G) > \epsilon'r$, for $r \geq R'$.

Define
\[
B := \bigcap_{k\geq 0} \{ z \in \C : \operatorname{Re} G^k(z) \geq \frac{1}{\epsilon'} M^k(\epsilon' R', G) \text{ or } G^k(z) \notin \mathcal{T} \}.
\]
%
Once again, $B$ is closed, and $B^c \subset \mathcal{T}$.

We shall prove that all components of the complement of $B$ are bounded. Suppose, by way of contradiction, that $B^c$ has an unbounded component, say $\Gamma$. It follows that there is a tract $T_0$ of $G$ such that $\Gamma \subset T_0$.
 
Since $\Gamma$ is unbounded and open, there exist points $z_0', z_0'', z_0''' \in T_0$ and a curve $\gamma_0 \subset \Gamma$ joining $z_0', z_0''$ and $z_0'''$ such that the following all hold.
\begin{itemize}
\item $\operatorname{Re} z_0' = a_0 > R'$.
\item $\operatorname{Re} z_0'' = \alpha_1 a_0$ and, in fact, $z_0'' \in L_{p_0, \alpha_1 a_0}$.
\item $\operatorname{Re} z_0''' = \alpha_2 a_0$.
\item $\gamma_0 \subset \{ z \in T_0 : \operatorname{Re} z \in [a_0, \alpha_2 a_0]\}$.
\end{itemize}
Note that the condition that $z_0'' \in L_{p_0, \alpha_1 a_0}$ follows by Lemma~\ref{inflemma} \eqref{inflemmaa}, with $z = z_0'$ and $a = \operatorname{Re} z_0''$, since $\Gamma$ is unbounded.

We next use induction to prove that there is a sequence of curves $(\gamma_k)_{k\geq 0}$ such that the following all hold for $k \geq 0$.
\begin{enumerate}[(a)]
\item $\gamma_{k+1} \subset G(\gamma_k)$. \label{ipropa}
\item $\gamma_k$ joins points  $z_k', z_k''$ and $z_k'''$. \label{ipropb}
\item $\operatorname{Re} z_k' = a_k$ satisfies $a_{k} > \frac{1}{\epsilon'} M(\epsilon' a_{k-1}, G)$ for $k \geq 1$. \label{ipropc}
\item $\operatorname{Re} z_k'' = \alpha_1 a_k$ and, in fact, $z_k'' \in L_{p_k, \alpha_1 a_k}$. \label{ipropd}
\item $\operatorname{Re} z_k''' = \alpha_2 a_k$. \label{iprope}
\item There is a tract $T_k$ such that $\gamma_k \subset \{ z \in T_k : \operatorname{Re} z \in [a_k, \alpha_2 a_k]\}$. \label{ipropf}
\end{enumerate}

We need to prove that if this claim holds for all $0 \leq j \leq k$, then it also holds for $j = k+1$. Assume, then, that curves with properties \eqref{ipropa}--\eqref{ipropf} have been constructed for $0 \leq j \leq k$. Note that it follows from \eqref{ipropa} and \eqref{ipropc} that if $\gamma'_k$ is the component of $G^{-k}(\gamma_k)$ contained in $\gamma_0$, then 
\[
\operatorname{Re} G^n(z) \geq \frac{1}{\epsilon'}M^k(\epsilon'R',G), \qfor z \in \gamma'_k \text{ and } 0 \leq n \leq k.
\]
Note that there is a tract $T_{k+1}$ such that $G(\gamma_k) \subset T_{k+1}$. 

\begin{figure}
\centering
\includegraphics[scale=0.4]{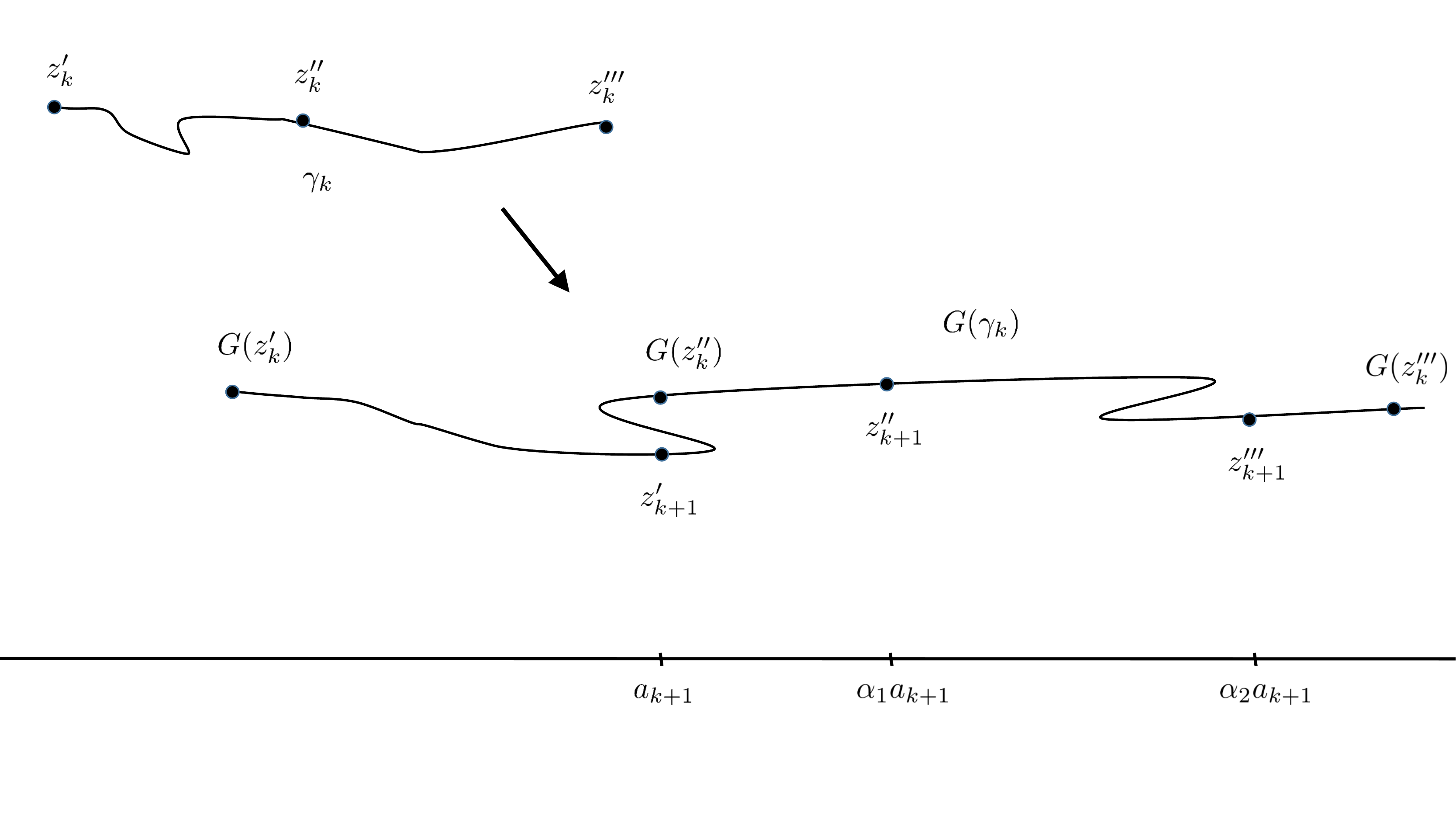}
\caption{\label{fig:lemma 5.1} The curves $\gamma_k$ and $G(\gamma_k)$}
\end{figure}

By an application of Lemma~\ref{growlemma}, we obtain that
\begin{equation*}
\operatorname{Re} G(z_k''') > \exp(\alpha_2 \epsilon a_k) \operatorname{Re} G(z_k'') > \alpha_2 \operatorname{Re} G(z_k'').
\end{equation*}
We let $z_{k+1}'$ be a point of $G(\gamma_k)$ of real part $a_{k+1}=\operatorname{Re} G(z_k'')$, and we let $z_{k+1}'''$ be a point of $G(\gamma_k)$ of real part $\alpha_2 a_{k+1} < \operatorname{Re} G(z_k''')$ (see Figure \ref{fig:lemma 5.1}). We now apply Proposition~\ref{prop:infnew} with $x_1= a_{k+1}$, $x_2= \alpha_2 a_{k+1}$ and $t=\alpha_1 a_{k+1}$ to deduce that there is a point $$z_{k+1}'' \in L_{p_{k+1}, \alpha_1 a_{k+1}} \cap G(\gamma_k).$$ It remains to show that \eqref{propc} is satisfied. \\

The argument is very similar to \cite[proof of Theorem 5.2]{DevHairs}, but we include some details for the reader's convenience. Set
\[
c_k = \frac{1}{2}\alpha_1 a_k = \frac{1}{2} \operatorname{Re} z_k''.
\]
Choose $w_k \in \gamma_k \cap L_{p_k, c_k}$. Once again this is possible because of Proposition~\ref{prop:infnew}.

Let $c = (4(\alpha+2))^{-1}$. It follows from the choice of $R'$, from Lemma~\ref{inflemma2}, and using an argument exactly as in \cite[p.756]{DevHairs} that
\[
\operatorname{Re} G(z_k'') > c |G(z_k'')|.
\]
  
By the choice of constants, we have that $c_k/D > p'$. It follows by Lemma~\ref{inflemma}~\eqref{inflemmab}, together with \eqref{Mintract}, that there is a point $w_k' \in L_{p_k, c_k}$ such that 
\[
\operatorname{Re} G(w_k') \geq M(c_k/D, G).
\]
We apply Lemma~\ref{inflemma2} again to obtain that
\[
\frac{|G(z_k'')|}{M(c_k/D, G)} \geq \frac{|G(z_k'')|}{|G(w_k')|} \geq \exp\left(\frac{1}{4}\alpha_1a_k - 4\pi\right).
\]

Hence
\[
\operatorname{Re} G(z_k'') > c \left|G(z_k'')\right|\geq c\exp\left(\frac{1}{4}\alpha_1a_k - 4\pi\right)M(c_k/D, G).
\]

It then follows from our choice of $\alpha_1$ that
\begin{equation}
\label{aclaim}
a_{k+1} = \operatorname{Re} G(z_k'') > \frac{1}{\epsilon'} M(\epsilon' a_k, G),
\end{equation}
which completes the construction. \\

By Lemma~\ref{RSlemma}, there is a point $\zeta \in \gamma_0$ such that 
\[
\operatorname{Re} G^n(\zeta) \geq \frac{1}{\epsilon'} M^n(\epsilon'R', G), \qfor n \in \N.
\]
This implies that $\zeta \in B$, which is a contradiction. This completes the proof that all components of the complement of $B$ are bounded. \\

We next claim that $B \subset A(G) \cup J(G)^c$. 
For, if $z \in B \cap J(G)$, then it follows from the choice of $R'$, together with the definition of $B$, that 
\[
\operatorname{Re} G^n(z) \geq \frac{1}{\epsilon'} M^n(\epsilon'R', G) \geq M^n(\epsilon'R', G), \qfor n \in \N,
\]
and so $z \in A(G)$ as required. \\

To complete the proof of the lemma, suppose that $z \in J(G) \setminus A(G)$. Then $z \in B^c$, and so we can let $X$ be the component of $B^c$ containing $z$. We know that this set is bounded. Since $B$ is closed, it follows that $\partial X$ is a continuum in $B \subset A(G) \cup J(G)^c$ which separates $z$ from infinity, as required.
\end{proof}
We are now able to prove Theorem \ref{theo:disjoint-type-inf-order}.
\begin{proof}[Proof of Theorem~\ref{theo:disjoint-type-inf-order}]
The first part of the theorem follows easily from Lemma~\ref{lemm:otherinf} together with the fact that $X(f) \subset I(f)$.

The proof of the second part is very similar to the proof of Theorem \ref{theo:disjoint-type}. First we use Theorem \ref{theo:disjoint-type-inf-order-context},  Theorem \ref{theo:implicationsofaconjugacy} and Proposition~\ref{prop:infnew} to show that every point in $\meanderingendpoints(f)$ can be separated from infinity by a continuum $\gamma \subset A(f) \cup F(f)$. 

Finally, since $J(f)$ is a Cantor bouquet we have that $\meanderingendpoints(f)\subset \endpoints(f)$ is totally separated. We can then use the argument from the proof of Theorem \ref{theo:disjoint-type} to show that $\meanderingendpoints(f)\cup\{\infty\}$ is totally separated. 
\end{proof}
%

%
%
%
\section{Proof of Theorem \ref{theo:Fatou-functions}}
\label{not B_sec}


It follows from \cite[Theorem 2]{semiconjugation} that $J(g) = \pi(J(f))$, and so $\pi$ is a local homeomorphism from $J(f)$ to $J(g)$. In particular, $\endpoints(g) = \pi(\endpoints(f))$. Note also that $A(g)\subset \pi (A(f))$ \cite[Theorem 5]{semiconjugation}. It follows by Theorem~\ref{theo:disjoint-type-context} and Theorem~\ref{theo:Fatou-functions-context} that $\meanderingendpoints(g) \supset \pi (\meanderingendpoints(f))$.

Suppose that $z \in \meanderingendpoints(f)$. We claim that there is a bounded continuum in $A(f) \cup F(f)$ which separates $z$ from infinity.

Let $w= \pi(z)$. Then, by the above, $w \in \meanderingendpoints(g)$. Hence, by the construction in the proof of Theorem \ref{theo:disjoint-type}, there exists a continuum $\gamma \subset A(g) \cup F(g)$ that separates $w$ from infinity. Since $0 \in F(g)$, we can assume that $0 \notin \gamma$. Let $X$ be the bounded complementary component of $\gamma$ which contains $w$.

We now have two cases. On one hand, suppose that $0 \notin X$. Then the component, $X'$ say, of $\pi^{-1}(X)$ that contains $z$ is bounded. We let $\gamma' = \partial X'$, so that $\gamma'$ is a continuum in $A(f)\cup F(f)$ that separates $z$ from infinity, as required. 

On the other hand, suppose that $0 \in X$. The preimage under $\pi$ of $\partial X$ is an unbounded continuum, which separates the preimage, $Y$, of $X$ (which includes $z$) from the rest of the plane.
Note that $Y$ contains a half-plane which is compactly contained in $F(f)$ (since $0$ is an attracting fixed point of $g$). Call this half-plane $Y'$. Note that the Fatou set of $f$ consists of one completely invariant Baker domain (see \cite[Remark 2 after Theorem 4.1]{DevHairs}).

Now, $\partial X$ cannot lie in $A(g)$, because, if so, then $\partial X$ lies in $J(g)$, and $F(g)$ is an unbounded domain that contains the origin. Hence $\partial X$ contains a point in $F(g)$, say $t'$. It follows that $\partial Y$ contains infinitely many (periodic) preimages of $t'$, say $t_n, n \in \mathbb{Z}$, each of which lies in $F(f)$.

Since $F(f)$ is connected, we can join $t_0$ to $\partial Y'$ by a simple curve $\omega_0\subset Y \cap F(f)$. The periodic copies of $\omega_0$ give a collection of curves in $Y \cap F(f)$, $(\omega_n)_{n\in\mathbb{Z}}$, each of which joins $\partial Y'$ to $t_n$. 

Set $$\Gamma' := \partial Y \cup \partial Y' \cup \bigcup_{j \in \mathbb{Z}} \omega_j.$$ Then $z$ lies in a bounded component, say $K$, of the complement of $\Gamma'$. Then $\partial K$ is the required bounded continuum in $A(f) \cup F(f)$ which separates $z$ from infinity. This completes the proof of the claim.

Since $J(f)$ is a Cantor bouquet, we have that $\endpoints(f)$ is totally separated. Hence the result follows by Lemma \ref{lemm:LasseVasso}.
%
%
%
\section{Hyperbolic functions}
\label{hyperbolic_sec}
To prove Theorem \ref{theo:hyperbolic-esc}, we require the following, which is part of \cite[Theorem~5.2]{Rigidity}. This establishes a  semiconjugacy between the Julia set of a hyperbolic function and the Julia set of a certain disjoint-type function.
\begin{theorem}
\label{theo:rigidityconjugacy}
Suppose that $f \in \B$ is hyperbolic. Then there exists $L>0$ such that the following holds. Set $g(z) = f(z/L)$. Then $g$ is of disjoint type, and there exists a continuous surjection $\vartheta: J(g) \to J(f)$ such that $f \circ \vartheta = \vartheta \circ g$. Moreover $\vartheta$ is a homeomorphism from $I(g)$ to $I(f)$, and $\vartheta(z) \rightarrow \infty$ as $z \rightarrow \infty$.
\end{theorem}
\begin{remark}\normalfont
The fact that $\vartheta(z) \rightarrow \infty$ as $z \rightarrow \infty$ is not stated in \cite[Theorem 5.2]{Rigidity}, but is an immediate consequence of \cite[Equation 5.2]{Rigidity}.
\end{remark}
We require the following additional property of the map $\vartheta$ in Theorem~\ref{theo:rigidityconjugacy}.
\begin{proposition}
\label{prop:vartheta}
Suppose that the maps $f, g$ and $\vartheta$ are as in the statement of Theorem~\ref{theo:rigidityconjugacy}. Then $\vartheta$ is a homeomorphism from $A(g)$ to $A(f)$.
\end{proposition}
\begin{remark}\normalfont
It does not seem possible to use Theorem~\ref{theo:implicationsofaconjugacy} here, since the function $\vartheta$ is only defined on the Julia set.
\end{remark}
\begin{proof}[Proof of Proposition~\ref{prop:vartheta}]
We show that $\vartheta(A(g)) \subset A(f)$. The proof of the reverse inclusion $A(f) \subset \vartheta(A(g))$ is very similar, and is omitted.

We note first, by \cite[equation (5.2)]{Rigidity}, that there is a constant $C>0$ and a domain $W$, which contains a neighbourhood of infinity, such that 
\begin{equation}
\label{Rigeq}
\operatorname{dist}_W(w, \vartheta^{-1}(w)) \leq C, \qfor w \in J(f).
\end{equation}
Here $\operatorname{dist}_W(w,z)$ denotes the hyperbolic distance in $W$ between points $w$ and $z$.

It can be shown, using for example \cite[Theorem 4]{Minda}, that there are constants $C', R_0 > 0$ such that the hyperbolic density in $W$, denoted by $\rho_W$, satisfies
\begin{equation}
\label{rhoeq}
\rho_W(w) \geq \frac{C'}{|w| \log|w|}, \qfor |w| > R_0.
\end{equation}

Without loss of generality, we can assume that $R_0$ is chosen sufficiently large that $|w| > R_0$ implies that $|\vartheta^{-1}(w)| > 1$. We claim that there is a constant $C'' \in (0,1]$ such that
\begin{equation}
\label{thetaandw}
|w| \geq |\vartheta^{-1}(w)|^{C''}, \qfor w \in J(f) \text{ such that } |w| > R_0.
\end{equation}

To prove this claim, suppose that $w \in J(f)$ and that $|w| > R_0$. We can assume that $|w| < |\vartheta^{-1}(w)|$, as otherwise there is nothing to prove. By \eqref{Rigeq} and \eqref{rhoeq}, we have
\[
C \geq \operatorname{dist}_W(w, \vartheta^{-1}(w)) \geq \int_{|w|}^{|\vartheta^{-1}(w)|} \frac{C'}{t \log t} \ dt = C' \log\frac{\log |\vartheta^{-1}(w)|}{\log |w|}.
\]
The claim then follows with $C'' = \exp(-C/C')$. 

We also have that $M(RL, g) = M(R, f)$, for $R > 0$. Hence if $R>R_0$ is sufficiently large that $M(R, f) > R$, then we can choose $R'>0$ such that 
\begin{equation}
\label{Mgrows}
M^n(R', g) > M^n(R, f), \qfor n \in \N.
\end{equation}

Now, suppose that $z \in A(g)$, and set $w = \vartheta(z)$. Note that $z \in I(g)$, and so $w \in I(f)$ because $\vartheta$ is a homeomorphism from $I(g)$ to $I(f)$. Without significant loss of generality, we can suppose that all iterates of $z$ under $g$, and all iterates of $w$ under $f$, are of modulus greater than $R$. Since $z \in A(g)$, there is an $\ell \in \N$ such that
\[
|g^{n+\ell}(z)| \geq M^n(R', g), \qfor n \in \N.
\]

Hence, by \eqref{Mgrows}, together with the fact that $f, g$ are conjugate,
\[
|\vartheta^{-1}(f^{n+\ell}(w))| \geq M^n(R', g) > M^n(R, f), \qfor n \in \N.
\]

It then follows by \eqref{thetaandw} that
\[
|f^{n+\ell}(w)| > (M^n(R, f))^{C''}, \qfor n \in \N.
\]

By well-known properties of the maximum modulus function, we can deduce from this that $w \in A(f)$, as required.
\end{proof}
\begin{proof}[Proof of Theorem \ref{theo:hyperbolic-esc}]
Let $g$ be as in the statement of Theorem~\ref{theo:rigidityconjugacy}. It follows from Theorem~\ref{theo:disjoint-type-context} and Theorem~\ref{theo:disjoint-type} that $$\meanderingendpoints(g) \cup \{\infty\} = J(g) \setminus A(g) \cup \{\infty\}$$ is totally separated. In particular $(I(g) \setminus A(g)) \cup \{\infty\}$ is totally separated, since, as is well-known, $I(g) \subset J(g)$. 

The map $\vartheta$ of Theorem~\ref{theo:rigidityconjugacy} is a homeomorphism from $I(g)$ to $I(f)$ and, by Proposition \ref{prop:vartheta}, is also a homeomorphism from $A(g)$ to $A(f)$. Set $\vartheta(\infty) = \infty$, so that $\vartheta$ is a homeomorphism
\[
\vartheta : (I(g) \setminus A(g)) \cup \{\infty\} \to (I(f) \setminus A(f)) \cup \{\infty\}.
\]

We can deduce, for example by \cite[Observation 2.2]{LasseNada}, that $(I(f) \setminus A(f)) \cup \{\infty\}$ is totally separated. The result then follows, since \cite[Theorem 1.2]{DevHairs} implies that $\escapingandmeanderingendpoints(f) = I(f) \setminus A(f)$.
\end{proof}
\begin{remark}\normalfont
It is well-known that $f \in \B$ is hyperbolic if and only if the \emph{postsingular set} $P(f)$ is compactly contained in the Fatou set. A larger class of functions are those $f \in \B$ such that $P(f) \cap F(f)$ is compact and $P(f) \cap J(f)$ is finite; these are known as \emph{subhyperbolic}. Mihaljevi\'c-Brandt \cite[Theorem 5.1]{HMB} showed that the semiconjugacy of \cite[Theorem~5.2]{Rigidity} can be constructed in this larger class. She also showed \cite[Corollary 1.3]{HMB} that if $f$ is \emph{strongly} subhyperbolic, then $J(f)$ is a pinched Cantor bouquet. It can then be deduced that the hypothesis that $f$ is hyperbolic in Theorem~\ref{theo:hyperbolic-esc} can be replaced by the weaker assumption that $f$ is strongly subhyperbolic. We refer to \cite{HMB} for precise definitions of terms in this remark.
\end{remark}
%
%
%
%
\section{Spiders' webs}
\label{spw_sec}
Before proving Theorem~\ref{theo:disjoint-type-spw}, we first prove the following. 
\begin{theorem}
\label{theo:spw}
Suppose that $f$ is a \tef, and that $X \subset \C$ has no bounded components, and is such that $f(X) \subset X$. If there is a bounded domain $G$ such that $G \cap J(f) \ne \emptyset$ and $\partial G \subset X$, then $X$ is a spider's web.
\end{theorem}
\begin{proof}
For each $n \in \N$, let $G_n= T(f^n(G))$. We show that there exists a subsequence $(n_k)_{k\in\N}$ such that 
\begin{equation}
\label{Gnsub3} 
G_{n_k} \subset G_{n_{k+1}} \text{ and } \partial G_{n_k} \subset X, \text{ for } k \in \mathbb{N}, \text{ and } \bigcup_{k\in\N} G_{n_k}= \mathbb{C}.
\end{equation}
 
To achieve this, let $(R_n)_{n\in\N}$ be a sequence of positive real numbers tending to infinity. We define the sequence $(n_k)_{k\in\N}$ inductively as follows. First we choose $n_1$ sufficiently large that $\{ z \in \C : |z| < R_1 \} \subset G_{n_1}$; this is possible because $G$ meets $J(f)$, and by the well-known ``blowing-up'' property of the Julia set.  Next, if $n_k$ is defined, then we choose $n_{k+1}$ sufficiently large that $$\{ z \in \C : |z| < R_{k+1} \}  \cup G_{n_k} \subset G_{n_{k+1}};$$ this is possible for the same reason.  

Now, for each $n \in \mathbb{N},$ $\partial G_n \subset \partial f^n(G) \subset f^n(\partial G),$ since $G$ is a bounded domain. Also, as $\partial G  \subset X$ and $f(X) \subset X$, we have that
\begin{equation*}
\partial G_n \subset f^n(\partial G) \subset X, \qfor n \in \N.
\end{equation*}
We can then deduce that \eqref{Gnsub3} holds.

The fact that $X$ is a spider's web can quickly be deduced from \eqref{Gnsub3}, together with the fact that $X$ has no bounded components.
\end{proof}

\begin{proof}[Proof of Theorem~\ref{theo:disjoint-type-spw}]
In the proof of Theorem~\ref{theo:disjoint-type} we showed that every point of $\meanderingendpoints(f)$ can be separated from infinity by a continuum $\gamma \in A(f) \cup F(f)$. Choose any point in $z \in \meanderingendpoints(f)$, and let $G$ to be the component of the complement of $\gamma$ containing $z$. Then $z \in J(f)$ and $\partial G \subset (A(f) \cup F(f))$.

It is well-known that $f(A(f)) \subset A(f)$ and $f(F(f)) \subset F(f)$. All components of $A(f)$ are unbounded, by \cite[Theorem 1]{qnofFandE}. Since $f$ is of disjoint-type, $F(f)$ has one component which is unbounded.

It follows by Theorem~\ref{theo:spw} that $A(f) \cup F(f)$ is a spider's web.
\end{proof}  

\section{An example}
\label{examples_sec}
\begin{figure}
  \subfloat[$J(g_{0.995})$]{\includegraphics[width=.7\textwidth,height=.5\textwidth]{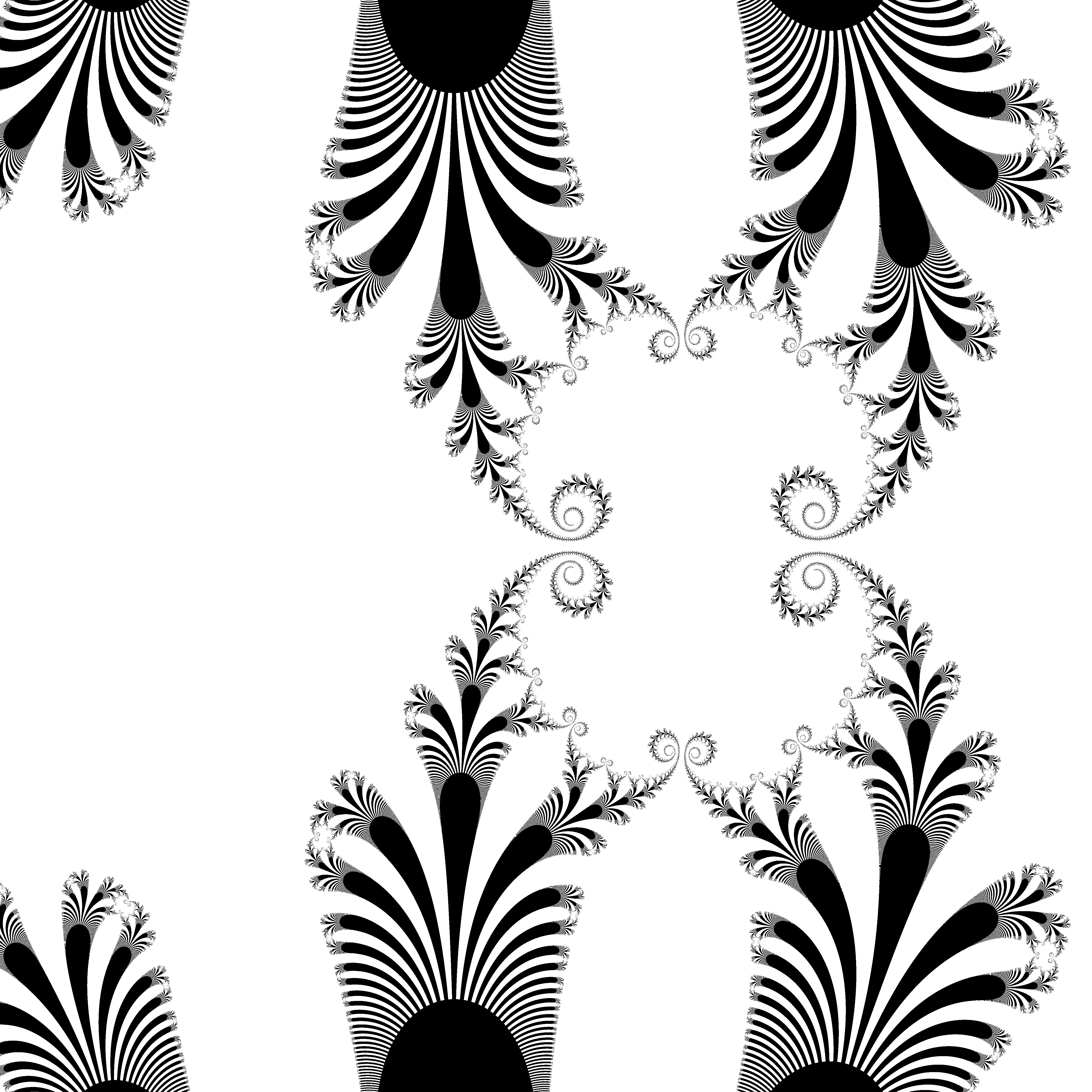}}\hfill
  \subfloat[$J(g_1)$]{\includegraphics[width=.7\textwidth,height=.5\textwidth]{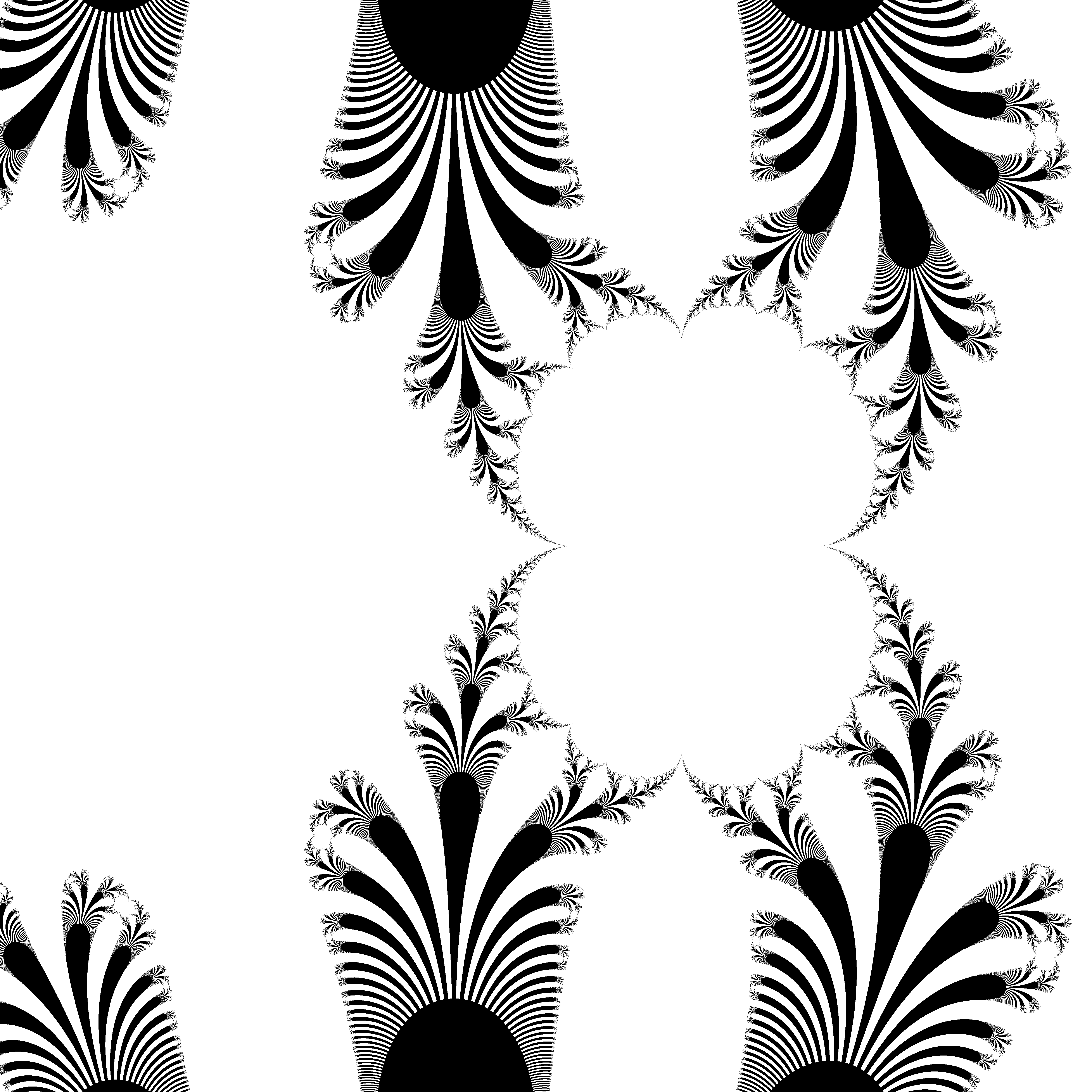}}\hfill
	\subfloat[$J(g_{1.1})$]{\includegraphics[width=.7\textwidth,height=.5\textwidth]{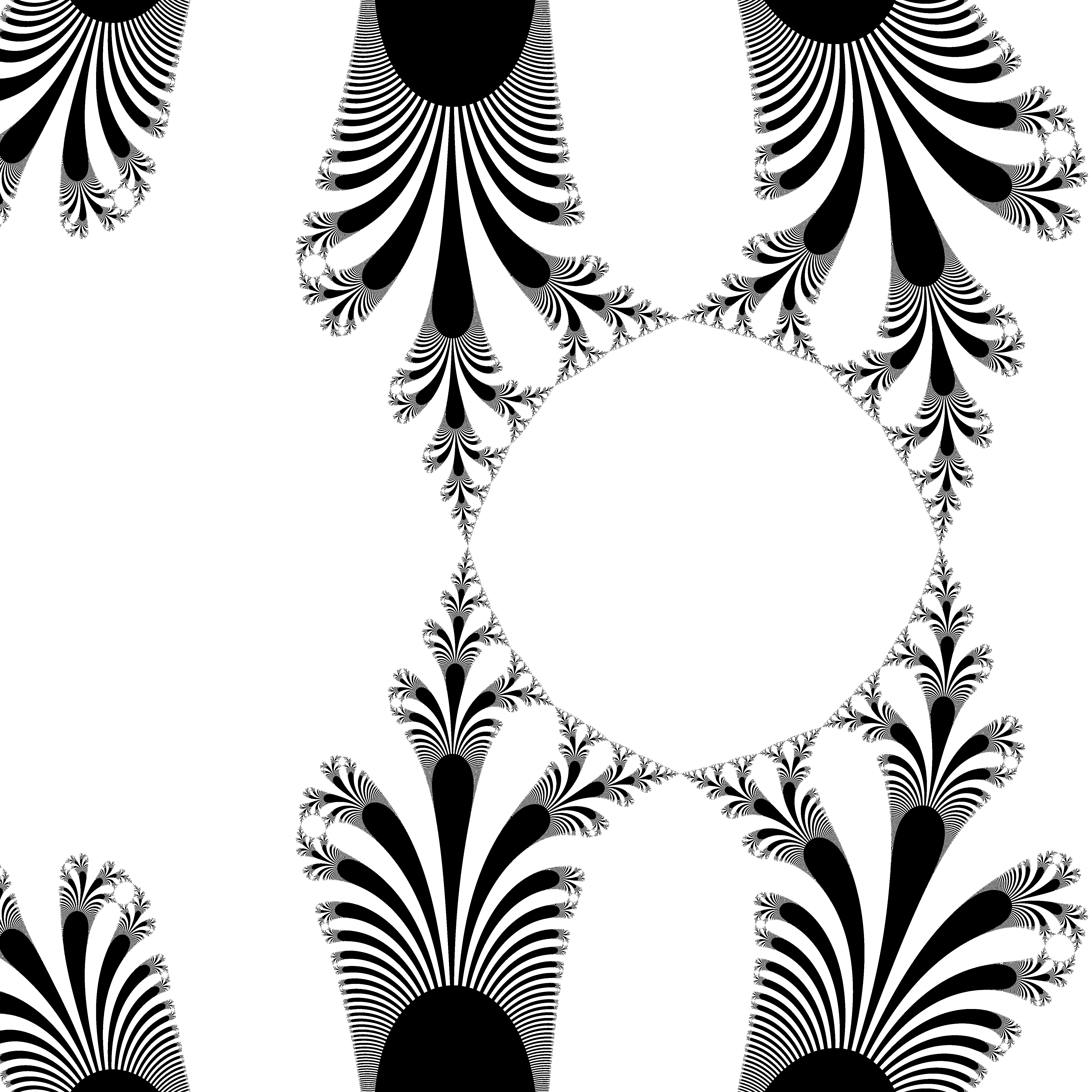}}
   \caption{\label{fig:figure} These figures show Julia sets, in black, for functions in the family $g_\lambda$ defined in \eqref{gdef}. In each the top left corner is at $-0.35+1.2i$ and the bottom right corner is at $2.23-1.2i$.}
\end{figure}

In this section we briefly discuss the dynamics of maps in the family $g_\lambda$. Recall that this family was defined in \eqref{gdef} by
\[
g_\lambda(z) := \lambda z^2 \exp(z - z^2), \qfor \lambda > 0.
\]
It is easy to see that these functions have finite order.

It can be shown by a calculation that if $\lambda \in (0, 1)$, then the function $g_{\lambda}$ is a disjoint type function, and so the hypotheses of Theorem~\ref{theo:disjoint-type} are satisfied; see the top image of Figure~\ref{fig:figure}.

Suppose next that $\lambda$ is close to, but greater than, one. In this case $g_\lambda$ has one asymptotic value, at the origin, and three critical values; zero, $p_\lambda > 1$, and $q_\lambda \in (0,1)$. There is a superattracting fixed point at the origin, and $q_\lambda$ lies in the (unbounded) immediate attracting basin of this point. There is also an attracting fixed point slightly greater than one, and $p_\lambda$ lies in the immediate attracting basin, $U$, of this point. It follows that $g_\lambda$ is hyperbolic. It also follows, by \cite[Theorem 1.10]{BFR}, that $U$ is a bounded Jordan domain (in fact a quasidisc). Hence the conclusion of Theorem~\ref{theo:disjoint-type} does not hold, since $\partial U \subset \meanderingendpoints(g_{\lambda})$; see the lower image of Figure~\ref{fig:figure}.

Finally, in the case that $\lambda = 1$, there is a parabolic fixed point at $1$, and $p_1$ lies in the immediate parabolic basin of $1$, $V$ say. It follows that $g_1$ is not hyperbolic. We note that this function lies in a class of functions called \emph{parabolic} in forthcoming work by Alhamd. Combining her work with techniques in the proof of \cite[Theorem 1.10]{BFR}, it seems likely that it can be shown that $V$ is bounded. Hence the conclusion of Theorem~\ref{theo:disjoint-type} does not hold; see the middle image of Figure~\ref{fig:figure}. \\

%
%
%
%
%
\emph{Acknowledgment:} The authors are grateful to Dan Nicks and Lasse Rempe-Gillen for many useful discussions, and to Simon Albrecht for a detailed reading of a draft.
%
%
%
%
%
%
\newpage
\bibliographystyle{acm}
\bibliography{References}
\end{document}